\documentclass[11pt]{article}
\usepackage{amsmath,amssymb,amsfonts,amsthm}
\usepackage{float}
\numberwithin{equation}{section}
\newtheorem{theorem}{Theorem}[section]

\newtheorem{lemma}[theorem]{Lemma}
\newtheorem{proposition}[theorem]{Proposition}
\newtheorem{corollary}[theorem]{Corollary}

\begin{document}
\begin{center}
{\Large{\textbf{Introduction to the McPherson number, $\Upsilon(G)$ of a simple connected graph}}} 
\end{center}
\vspace{0.5cm}
\large{\centerline{(Johan Kok, Susanth C)\footnote {\textbf {Affiliation of author(s):}\\
\noindent Johan Kok (Tshwane Metropolitan Police Department), City of Tshwane, Republic of South Africa\\
e-mail: kokkiek2@tshwane.gov.za\\ \\
\noindent Susanth C (Department of Mathematics, Vidya Academy of Science and Technology), Thalakkottukara, Thrissur-680501, Republic of India\\
e-mail: susanth\_c@yahoo.com}}
\vspace{0.5cm}
\begin{abstract}
\noindent The concept of the \emph{McPherson number} of a simple connected graph $G$ on $n$ vertices denoted by $\Upsilon(G),$ is introduced. The recursive concept, called the \emph{McPherson recursion}, is a series of \emph{vertex explosions} such that on the first iteration a vertex $v \in V(G)$ explodes to arc (directed edges) to all vertices $u \in V(G)$ for which the edge $vu \notin E(G),$ to obtain the mixed graph $G'_1.$ Now $G'_1$ is considered on the second iteration and a vertex $w \in V(G'_1) = V(G)$ may explode to arc to all vertices $z \in V(G'_1)$ if edge $wz \notin E(G)$ and arc $(w,z)$ or $(z,w) \notin E(G'_1).$ The \emph{McPherson number} of a simple connected graph $G$ is the minimum number of iterative vertex explosions say $\ell$, to obtain the mixed graph $G'_\ell$ such that the underlying graph of $G'_\ell$ denoted $G^*_\ell$ has $G^*_\ell \simeq K_n.$ We determine the \emph{McPherson number} for paths, cycles and $n$-partite graphs. We also determine the \emph{McPherson number} of the finite Jaco Graph $J_n(1), n \in \Bbb N.$ It is hoped that this paper will encourage further exploratory research.
\end{abstract}
\noindent {\footnotesize \textbf{AMS Classification Numbers:} 05C07, 05C20, 05C38, 05C75, 05C85} 
\section{Introduction} 
The concept of the \emph{McPherson number} of a simple connected graph $G$ on $n$ vertices denoted by $\Upsilon(G),$  is introduced. The recursive concept, called the \emph{McPherson recursion}, is a series of \emph{vertex explosions} such that on the first iteration a vertex $v \in V(G)$ explodes to arc (directed edges) to all vertices $u \in V(G)$ for which the edge $vu \notin E(G),$ to obtain the mixed graph $G'_1.$ Now $G'_1$ is considered on the second iteration and a vertex $w \in V(G'_1) = V(G)$ may explode to arc to all vertices $z \in V(G'_1)$ if edge $wz \notin E(G)$ and arc $(w,z)$ or $(z,w) \notin E(G'_1).$ The \emph{McPherson number} of a simple connected graph $G$ is the minimum number of iterative vertex explosions say $\ell$, to obtain the mixed graph $G'_\ell$ such that the underlying graph $G^*_\ell \simeq K_n.$ Note that $G = G^*_0.$\\ \\
It is easy to see that the total number of arcs created is $\epsilon(K_n) - \epsilon(G).$ It is equally easy to see that $\Upsilon(K_n) = 0$ and $\Upsilon(K_n - uv)_{uv \in E(K_n)} = 1.$ It is not that easy to see that the sequence of vertex explosions does not generally obey the commutative law. This will be illustrated by way of an example.\\ \\
\noindent \textbf{Example 1.} Let $G$ be the simple connected graph with $V(G) =\{v_1, v_2, v_3, v_4, v_5, v_6\}$ and $E(G) = \{v_1v_2, v_1v_3, v_1v_4, v_1v_5, v_2v_5, v_2v_6, v_3v_4, v_3v_5\}.$\\ \\
On the first iteration let vertex $v_6$ explode to create the arcs $(v_6,v_1), (v_6, v_3), (v_6, v_4), (v_6, v_5).$ On the second iteration let $v_2 \in V(G'_1)$ explode to create arcs $(v_2, v_3), (v_2, v_4).$ Finally let $v_4 \in V(G'_2)$ explode to create arc $(v_4, v_5).$ After validating all permutated sequences of vertex explosions we conclude that, since three explosions are the minimum needed to ensure that the underlying graph $G^*_3 \simeq K_6$, it follows that $\Upsilon(G) = 3.$\\ \\
However, if vertex $v_1 \in V(G)$ explodes first, the arc $(v_1, v_6)$ is created. If secondly, the vertex $v_3 \in V(G'_1)$ explodes the arcs $(v_3, v_2)$ and $(v_3, v_6)$ are created. Followed by the next iteration let $v_2 \in V(G'_2)$ explode to create arc $(v_2, v_4)$ followed by the explosion of vertex $v_5 \in V(G'_3)$ to create arcs $(v_5, v_4), (v_5, v_6).$  The final vertex explosion is that of vertex $v_4$ to create the arc $(v_4, v_6)$. Now only is the underlying graph $G^*_5 \simeq K_6.$ Since, minimality is defined we see that the sequence of explosions does not generally obey the commutative law.\footnote{Whilst listening to an amasing djembe drumming group from Ghana on Friday, 10 October 2014, celebrating the $15^{th}$ anniversary of Klitsgras Drumming Circle, the concept of McPherson numbers struck Kokkie's mind. Thank you to them. It reminds us there are mathematics in music and vice versa.}
\section{McPherson numbers of Jaco Graphs, Paths, Cycles and n-Partite Graphs}
\noindent Before we consider specialised graphs we will endeavour to find the optimal algorithm to calculate the McPherson number, $\Upsilon(G)$ of a simple connected graph $G$. The \emph{McPherson recursion} described in Lemma 2.2 below is sufficient.
\begin{lemma}
The McPherson number of a simple connected graph $G$ on n vertices is\\ $\Upsilon(G) \leq \epsilon(K_n) - \epsilon(G).$
\end{lemma}
\begin{proof}
Consider any complete graph $K_n, n \in \Bbb N.$ Without loss of generality let $n$ be even and let $G = K_n - (v_1v_2, v_3v_4, v_5v_6, ..., v_{n-1}v_n).$ Clearly $\epsilon(K_n) - \epsilon(G) = \frac{n}{2}.$ For any vertex $v_i$ eligible to explode only one arc, either $(v_i, v_{i+1})$ or $(v_i, v_{i-1})$ is added. Hence exactly $\frac{n}{2}$ vertex explosions are required to have $G^*_{\frac{n}{2}} \simeq K_n.$ Therefore, $\Upsilon(G) = \frac{n}{2} = \epsilon(K_n) - \epsilon(G).$ For all other simple connected graphs on $n$ vertices, $\Upsilon(G) \leq \epsilon(K_n) - \epsilon(G)$ since vertex explosions are defined to be \emph{greedy} and on each iteration always arcs to the maximum number of non-adjacent vertices in $G^*_i$ on the $(i+1)^{th}$-iteration.
\end{proof}
\begin{lemma}
The McPherson number of a simple connected graph $G$ on $n$ vertices is obtained through the McPherson recursion. Let any vertex with degree equal to $\delta(G)$ explode on the first iteration. This is followed by letting any vertex in the underlying graph $G^*_1$ of $G'_1$ with degree equal to $\delta(G^*_1)$ explode on the second iteration to obtain $G'_2$ and, so on. If after exactly $\ell$ vertex explosions the underlying graph $G^*_\ell \simeq K_n$ then, $\Upsilon(G) = \ell.$ 
\end{lemma}
\begin{proof}
Consider any simple connected graph $G$ on $n$ vertices. Apply the \emph{McPherson recursion} and assume after exactly $\ell$ vertex explosions the underlying graph $G^*_\ell \simeq K_n.$ Label the vertices which exploded consecutively $v_1, v_2, v_3, ..., v_\ell.$ The  total number of edges added in $G^*_\ell$ (total number of arcs added during the $\ell$ vertex explosions) is given by:\\ \\ $\epsilon(K_n) - \epsilon(G) = d_{G^*_0}^+(v_1) + d_{G^*_1}^+(v_2) + d_{G^*_2}^+(v_3) + ... + d_{G^*_\ell}^+(v_{\ell-1}).$\\ \\
Now assume that for any vertex $v_i$ a vertex $v_j, d_{G^*_{i-1}}(v_j) > d_{G^*_{i-1}}(v_i)$ exploded instead. Then it follows that:\\ \\ 
$t = d_{G^*_0}^+(v_1) + d_{G^*_1}^+(v_2) + d_{G^*_2}^+(v_3) + ...  + d_{G^*_{i-1}}^+(v_j) + ... + d_{G^*_\ell}^+(v_{\ell-1}) < d_{G^*_0}^+(v_1) + d_{G^*_1}^+(v_2) + d_{G^*_2}^+(v_3) + ...+ d_{G^*_{i-1}}^+(v_i) + ... + d_{G^*_\ell}^+(v_{\ell-1}) = \epsilon(K_n) - \epsilon(G).$ \\ \\ 
It means at least one more vertex explosion is needed to finally add the additional $(\epsilon(K_n) - \epsilon(G)) -t$ arcs required to obtain $G^*_{\ell+} \simeq K_n.$  The latter is a contradiction in respect of minimality as defined. Hence the \emph{McPherson recursion} is well-defined and $\Upsilon(G) = \ell.$
\end{proof}
\subsection{The McPherson number of Jaco Graphs $J_n(1), n \in \Bbb N.$}
The infinite directed Jaco graph (\emph{order 1}) was introduced in $[3],$ and defined by $V(J_\infty(1)) = \{v_i| i \in \Bbb N\}$, $E(J_\infty(1)) \subseteq \{(v_i, v_j)| i, j \in \Bbb N, i< j\}$ and $(v_i,v_ j) \in E(J_\infty(1))$ if and only if $2i - d^-(v_i) \geq j.$ The graph has four fundamental properties which are; $V(J_\infty(1)) = \{v_i|i \in \Bbb N\}$ and, if $v_j$ is the head of an edge (arc) then the tail is always a vertex $v_i, i<j$ and, if $v_k,$ for smallest $k \in \Bbb N$ is a tail vertex then all vertices $v_ \ell, k< \ell<j$ are tails of arcs to $v_j$ and finally, the degree of vertex $k$ is $d(v_k) = k.$ The family of finite directed graphs are those limited to $n \in \Bbb N$ vertices by lobbing off all vertices (and edges arcing to vertices) $v_t, t > n.$ Hence, trivially we have $d(v_i) \leq i$ for $i \in \Bbb N.$\\ \\
Note that the \emph{McPherson recursion} can naturally be extended to directed graphs. For the Jaco graph $J_n(1), n \in \Bbb N, n \geq 3$ it easily follows that if the lowest indiced vertex $v_i$ for which the edge $v_iv_n$ exists is found, then $\Upsilon(J_n(1)) = i-1$. The next theorem presents a stepwise closed formula for $\Upsilon(J_n(1)), n \in \Bbb N, n \geq3.$ 
\begin{theorem} 
Consider the Jaco graph $J_n(1), n \in \Bbb N, n \geq 3.$ If $v_i$ is the prime Jaconian vertex we have:\\
\begin{equation*} 
\Upsilon(J_n(1))
\begin{cases}
=  i , &\text {if the edge $v_iv_n \notin E(J_n(1)$,}\\ \\ 
=  i-1, &\text {otherwise.}
\end{cases}
\end{equation*} 
\end{theorem}
\begin{proof}
(a) If the edge $v_iv_n$ exists the largest complete subgraph of $J_n(1)$ is given by $\Bbb H_n(1) + v_i \simeq K_{(n-i) + 1}, [3].$ From the definition of a Jaco Graph it follows that vertices $v_1, v_2, v_3, ..., v_{i-1}$ are non-adjacent to at least vertex $v_n.$ So exactly $i-1$ vertex explosions are required to obtain $J^*_{n,(i-1)} \simeq K_n.$ Hence, $\Upsilon(J_n(1)) = i-1.$\\ \\
(b) If the edge $v_iv_n \notin E(J_n(1))$ then the Hope graph [3], is the largest complete subgraph of $J_n(1).$ From the definition of a Jaco graph it follows that vertices $v_1, v_2, v_3, ..., v_i$ are non-adjacent to at least vertex $v_n.$ So exactly $i$ vertex explosions are required to obtain $J^*_{n,i} \simeq K_n.$ Hence, $\Upsilon(J_n(1)) = i.$
\end{proof}
\noindent Table 1 shows the $\Upsilon$-values for $J_n(1), 3 \leq n \leq 15.$ The table can easily be verified and extended by using the Fisher Algorithm [3]. Note that the Fisher Algorithm determines $d^+(v_i)$ on the assumption that the Jaco Graph is always sufficiently large, so at least $J_n(1), n \geq i+ d^+(v_i).$ For a smaller graph the degree of vertex $v_i$ is given by $d(v_i)_{J_n(1)} = d^-(v_i) + (n-i).$ In [3] Bettina's theorem describes an arguably, closed formula to determine $d^+(v_i)$. Since $d^-(v_i) = n - d^+(v_i)$ it is then easy to determine $d(v_i)_{J_n(1)}$ in a smaller graph $J_n(1), n< i + d^+(v_i).$\\ \\  
\textbf{Table 1}\\ \\
\begin{tabular}{|c|c|c|c|c|}
\hline
$i\in{\Bbb{N}}$&$d^-(v_i)$&$d^+(v_i)$&Prime Jaconian vertex,$v_j$&$\Upsilon(J_i(1))$\\
\hline
3&1&2&$v_2$&1\\
\hline
4&1&3&$v_2$&2\\
\hline
5&2&3&$v_3$&2\\
\hline
6&2&4&$v_3$&3\\
\hline
7&3&4&$v_4$&3\\
\hline
8&3&5&$v_5$&4\\
\hline
9&3&6&$v_5$&5\\
\hline
10&4&6&$v_6$&5\\
\hline
11&4&7&$v_7$&6\\
\hline
12&4&8&$v_7$&7\\
\hline
13&5&8&$v_8$&7\\
\hline
14&5&9&$v_8$&8\\
\hline
15&6&9&$v_9$&8\\
\hline
\end{tabular}\\ \\ \\ 
\textbf{Conjecture:} For a Jaco Graph $J_n(1), n \in \Bbb N, n \geq 3$ we have that $d^+(v_n)$ is unique (\emph{non-repetitive}) if and only if $\Upsilon(J_n(1))$ is unique (\emph{non-repetitive}).
\subsection{McPherson number of Paths, $P_n, n \in \Bbb N$}
\begin{proposition}
The McPherson number of a path $P_n, n \geq 3$ is given by $\Upsilon(P_n) = n-2.$
\end{proposition}
\begin{proof}
Consider any path $P_n, n\in \Bbb N$ and label the vertices from left to right, $v_1, v_2, v_3, ..., v_n.$ Clearly $d(v_1) = d(v_n) = \delta(P_n).$ So, without loss of generality let vertex $v_1$ explode on the first iteration. The arcs $(v_1, v_3), (v_1, v_4), (v_1, v_5), ..., (v_1, v_n)$ are added. In the graph $P'_{n,1}$ we have that $d(v_2) = d(v_n) = 2 = \delta(P^*_{n,1}).$ Without loss of generality let vertex $v_2$ explode in the second iteration. Now arcs $(v_2, v_4), (v_2, v_5), ..., (v_2, v_n)$ are added. Recursively, all vertices $v_3, v_4, v_5, ..., v_{n-2}$ must explode to have $P^*_{n, (n-2)} \simeq K_n.$\\ \\
Clearly the number of explosions are a minimum to ensure $P^*_{n, (n-2)} \simeq K_n$ so we have,\\ $\Upsilon(P_n) = n-2.$
\end{proof}
\subsection{McPherson number of Cycles, $C_n, n \in \Bbb N$}
\noindent We begin this subsection with an interesting lemma.
\begin{lemma}
Consider two simple connected graphs $G$ and $H$ with $G \neq H.$ If after minimum recursive explosions, say $t$ explosions in respect of vertices $v \in V(G)$, we have that $G^*_t \simeq H,$ then $\Upsilon(G) = \Upsilon(H) + t.$
\end{lemma}
\begin{proof}
Vertex explosions as defined, only add \emph{arcs}. So the mere fact that after the minimum $t$ explosions of $t$ vertices of $G$ we have that $G^*_t \simeq H$ implies that $\nu(G) = \nu(H).$ The definition of the McPherson number is well-defined [see Lemma 2.2], so  for all graphs $G$ and $H$ each on $n$ vertices, if $G \simeq H$ then $\Upsilon(G) = \Upsilon(H).$\\ \\
From the \emph{McPherson recursion} it follows that $\Upsilon(G) = \Upsilon(G^*_t) + t = \Upsilon(H) + t.$
\end{proof}
\begin{proposition}
The McPherson number of a cycle $C_n, n \geq 4$ is given by $\Upsilon(C_n) = n-2.$
\end{proposition}
\begin{proof}
Consider any path $P_n$ and the cycle $C_n$. Label the vertices of the path from left to right, $v_1, v_2, v_3, ..., v_n$ and the vertices of the cycle clockwise, $v_1, v_2, v_3, ..., v_n.$ Without loss of generality let vertex $v_1$ of both the path and the cycle explode on the first interation. It follows that $P'_{n,1}$ has amongst others, the arc $(v_1, v_n)$ whilst $C'_{n,1}$ has amongst others, the edge $v_1v_n.$ So it follows that $P^*_{n,1} \simeq C^*_{n,1}.$\\ \\
Hence, the result $\Upsilon(P_n) = \Upsilon(P^*_{n,1}) + 1 = \Upsilon(C^*_{n,1}) + 1 = \Upsilon(C_n)$ holds true. Therefore $\Upsilon(C_n) = n-2.$
\end{proof}
\subsection{McPherson number of n-Partite graphs, $K_{(n_1, n_2, ..., n_\ell)}, n_{i, \forall i} \in \Bbb N$}
\noindent We recall that the sequence of vertex explosions do not generally obey the commutative law. However, $n$-partite graphs are a class of graphs which does obey the commutative law. It follows because the vertices of $K_{(n_1, n_2, ..., n_\ell)}, n_{i, \forall i} \in \Bbb N$ can be partioned in $\ell$ subsets of \emph{pairwise non-adjacent} vertices. Although we mainly consider simple connected graphs, we will as a special case consider the following lemma for the \emph{edgeless} graph on $n$ vertices which we call the $n$-Null graph, denoted $\aleph_n$. It is important to note that $\aleph_n = \cup_{n-times}K_1.$ The importance lies in the fact that if we determine the \emph{McPherson number} of $\cup_{\forall i}G_i$ the explosion of  a vertex $v_j \in V(G_k)$ will arc to all vertices $v_s \in V(G_k)$, non-adjacent to $v_j$ as well as, arc to all vertices $v_t \in V(G_m)$ for all $m \neq k$.
\begin{lemma}
For any $n$-Null graph, $n \in \Bbb N$ we have that:\\ \\
(a) The vertex explosion sequence obeys the commutative law,\\
(b) The McPherson number is given by $\Upsilon(\aleph_n) = n-1.$
\end{lemma}
\begin{proof}
(a) Consider any $n$-Null graph, $n \in \Bbb N.$ Clearly the $n$ vertices can randomly be labelled $v_1, v_2, v_3, ..., v_n$ in $n!$ ways. Label these $n!$ \emph{vertex labelled} $n$-Null graphs, $\aleph_{(n,1)}, \aleph_{(n, 2)}, \aleph_{(n, 3)}, ..., \aleph_{(n,n!)}.$ Clearly, $\aleph_{(n,1)} \simeq \aleph_{(n, 2)} \simeq \aleph_{(n, 3)} \simeq ... \simeq \aleph_{(n,n!)}.$ The different random labelling represents the random (commutative law) vertex explosions and from Lemma 2.2 it follows that $\Upsilon(\aleph_{(n,1)}) = \Upsilon( \aleph_{(n, 2)}) = \Upsilon(\aleph_{(n, 3)}) = ... = \Upsilon(\aleph_{(n,n!)}).$\\ \\
(b) Consider anyone of the $n!$ \emph{vertex labelled} $n$-Null graphs say, $\aleph_{(n,i)}.$ Let vertex $v_1$ explode on the first iteration to add arcs $(v_1, v_2), (v_1, v_3), ..., (v_1, v_n).$ Then let vertex $v_2$ explode on the second iteration to add arcs $(v_2, v_3), (v_2, v_4), ..., (v_2, v_n).$ Clearly on the \emph{$i^{th}$}-iteration the arcs $(v_i, v_{i+1}), (v_i, v_{i+2}), ..., (v_i, v_n)$ are added. It implies that on the \emph{$(n-1)^{th}$}-iteration the last arc $(v_{n-1}, v_n)$ is added, to obtain $\aleph^*_{(n,i),(n-1)} \simeq K_n.$ Since the number of vertex explosions are a minimum we conclude that $\Upsilon(\aleph_{(n,i)}) = n-1.$ Following from part (a) we have that $\Upsilon(\aleph_n) = n-1$ in general.
\end{proof}
\begin{proposition}
For the n-Partite graphs, $K_{(n_1, n_2, ..., n_\ell)}, n_{i, \forall i} \in \Bbb N$ we have that\\ $\Upsilon(K_{(n_1, n_2, ..., n_\ell)}) = \sum\limits^{\ell}_{i=1} n_i - \ell.$
\end{proposition}
\begin{proof}
In the graph $G \uplus H$ we restrict the explosion of vertex, $v_i \in V(G)$ to arc to non-adjacent vertices $v_j \in V(G)$ and the arcing of the explosion of vertex $u_i \in V(H)$ to non-adjacent vertices $u_j \in V(H).$ Clearly $\Upsilon(G\uplus H) = \Upsilon(G) + \Upsilon(H) = \Upsilon(H) + \Upsilon(G).$\\ \\
Now label the $n_1$ vertices, $v_{1,1}, v_{1,2}, v_{1,3}, ..., v_{1,n_1},$ and label the $n_2$ vertices, $v_{2,1}, v_{2,2}, v_{2,3}, ..., v_{2,n_2}$ and so on, and finally label the $n_\ell$ vertices, $v_{\ell,1}, v_{\ell,2}, v_{\ell,3}, ..., v_{\ell,\ell}.$ In the $n$-partite graph we have that the edge $v_{i,j}v_{k,m}$ exists for $1 \leq i,j \leq \ell, 1 \leq j \leq n_i, 1 \leq m \leq n_k$ and $ i \neq  j.$ It implies that $\Upsilon(K_{(n_1, n_2, ..., n_\ell)}) = \Upsilon(\aleph_{n_1} \uplus \aleph_{n_2} \uplus \aleph_{n_3} \uplus ... \uplus \aleph_{n_\ell}).$\\ \\
Now we have that $\Upsilon(\aleph_{n_1} \uplus \aleph_{n_2} \uplus \aleph_{n_3} \uplus ... \uplus \aleph_{n_\ell}) = \Upsilon(\aleph_{n_1}) + \Upsilon(\aleph_{n_2}) + \Upsilon(\aleph_{n_3}) + ... + \Upsilon(\aleph_{n_\ell}) = (n_1 -1) + (n_2 -1) + (n_3 -1) + ... + (n_\ell -1) = \sum\limits^{\ell}_{i=1} n_i - \ell.$
\end{proof}
\noindent We note from the proof above that the \emph{McPherson recursion} could be relax in the sense that it applies to each $\aleph_i$ but not to $K_{(n_1, n_2, ..., n_\ell)}, n_{i, \forall i} \in \Bbb N$ as a singular graph. The commutativity of the $\uplus$ operation allows for the relaxation. It makes the following generalisation possible.
\begin{corollary}
Consider the simple connected graphs $G_1, G_2, G_3, ..., G_n$ and define the $n_G$-partite graph to be the graph $G_{n_G}$ obtained by adding all the edges $vu$, if and only if $v \in V(G_i)$ and $u \in G_j, i \neq j$ to $\cup_{\forall i}G_i.$ We have that $\Upsilon(G_{n_G}) = \sum\limits_{\forall i}\Upsilon(G_i).$
\end{corollary}
\begin{proof}
Similar to the proof of Proposition 2.8.
\end{proof} 
\noindent [Open problem: In the graph $G \uplus H$ we restrict the explosion of  vertex $v_i \in V(G)$ to arc to non-adjacent vertices $v_j \in V(G)$ and the arcing of the explosion of vertex $u_i \in V(H)$ to non-adjacent vertices $u_j \in V(H).$ Clearly $\Upsilon(G\uplus H) = \Upsilon(G) + \Upsilon(H) = \Upsilon(H) + \Upsilon(G).$ This law is called the $\uplus$-commutative law of \emph{McPherson numbers}. It is easy to see that the $\uplus$-associative law of \emph{McPherson numbers} namely,\\ \\
$\Upsilon(G \uplus H \uplus M) = \Upsilon(G \uplus H) +  \Upsilon(M) = \Upsilon(G) + \Upsilon(H \uplus M) = \Upsilon(G \uplus M) + \Upsilon(H)$ is valid as well.\\ 
Alternatively stated, $\Upsilon(G \uplus H \uplus M) = (\Upsilon(G) + \Upsilon( H)) +  \Upsilon(M) = \Upsilon(G) + (\Upsilon(H) +  \Upsilon(M)) = (\Upsilon(G) + \Upsilon(M)) + \Upsilon(H).$ If both graphs $G$ and $H$ are graphs on $n$ vertices we note that $\Upsilon(G\uplus K_n) = \Upsilon(G) + \Upsilon(K_n) = \Upsilon(G) + 0 = \Upsilon(G) = 0 + \Upsilon(G) = \Upsilon(K_n) + \Upsilon(G) = \Upsilon( K_n \uplus G).$ If possible, describe the algebraic structure.] \\
\noindent [Open problem: If possible, prove that for a Jaco Graph $J_n(1), n \in \Bbb N, n \geq 3$ we have that $d^+(v_n)$ is unique (\emph{non-repetitive}) if and only if $\Upsilon(J_n(1))$ is unique (\emph{non-repetitive}).]\\ 
\noindent [Open problem: Define the \emph{McPherson graph} to be the directed subgraph of $G'_{\Upsilon(G)}$ which was obtained through \emph{vertex explosions}.  Characterise the \emph{McPherson graph} in general if possible, or for some specialised graphs.]\\
\noindent [Open problem: Prove the conjecture that if graph $G$ on $n$ vertices has $d(v_i) = d(v_j), i \neq j$ for all vertices of $G$, it is always possible to add a vertex $v_{n+1}$ with edges such that $G+ v_{n+1}$ has the same graph theoretical structural properties.]\\ 
\noindent [Open problem: Example 1 showed that the graph $G$ has $\Upsilon(G) = 3.$ The maximum number of vertex explosions to obtain $G^*_5 \simeq K_6$ was given by $\Upsilon^*(G) = 5.$ The \emph{McPherson discrepancy} is defined to be $\Upsilon_d(G) = \Upsilon^*(G) - \Upsilon(G).$\ If $\Upsilon_d(G) =0$ the graph $G$ is said to be \emph{McPherson stable}. It easily follows that $K_n, C_4$ are \emph{McPherson stable} whilst $P_n, n\geq 4$ is not.  The maximum number of vertex explosions can be determine by applying the \emph{inverse McPherson recursion}.
The maximum McPherson number of a simple connected graph $G$ on $n$ vertices is obtained through the inverse McPherson recursion. Let any vertex with maximum degree $d(v) \leq n-2$ explode on the first iteration. This is followed by letting any vertex in the underlying graph $G^*_1$ of $G'_1$ with maximum degree $d(u) \leq n-2$ explode on the second iteration to obtain $G'_2$ and, so on. If after exactly $\ell^*$ vertex explosions the underlying graph $G^*_{\ell^*} \simeq K_n$ then $\Upsilon^*(G) = \ell^*.$ 
Characterise \emph{McPherson stable} graphs.]\\ 
\noindent [Open problem: Platonic graphs are the graphs whose vertices and edges are the vertices and edges of platonic solids such as \emph{the tetrahedron, the octahedron, the cube, the icosahedron, the dodecahedron} and alike, See [2].} Determine the McPherson number for platonic graphs.]\\ \\
\noindent \textbf{\emph{Open access:\footnote {To be submitted to the \emph{Pioneer Journal of Mathematics and Mathematical Sciences.\\  Dedicated to our friend, \emph{Vic McPherson}.}}}} This paper is distributed under the terms of the Creative Commons Attribution License which permits any use, distribution and reproduction in any medium, provided the original author(s) and the source are credited. \\ \\
References (Limited) \\ \\
\noindent $[1]$ Bondy, J.A., Murty, U.S.R., \emph {Graph Theory with Applications,} Macmillan Press, London, (1976). \\
\noindent $[2]$ Fr\'echet, M., Ky, F., \emph {Initiation to Combinatorial Topology,} Prindle, Weber and Schmidt, Boston, (1967). \\
\noindent $[3]$ Kok, J., Fisher, P., Wilkens, B., Mabula, M., Mukungunugwa, V., \emph{Characteristics of Finite Jaco Graphs, $J_n(1), n \in \Bbb N$}, arXiv: 1404.0484v1 [math.CO], 2 April 2014. \\
\end{document}